\documentclass[]{amsart}

\usepackage{amssymb}           
\usepackage{bm}                

\usepackage{amsmath}           
\usepackage{amsthm}            
\usepackage{amscd}             

\usepackage{ascmac}            
\usepackage{indentfirst}       

\usepackage{graphicx}          

\newtheorem{theorem}{Theorem}[section]
\newtheorem{prop}[theorem]{Proposition}

\newtheorem{lem}[theorem]{Lemma}

\theoremstyle{definition}
\newtheorem{defi}[theorem]{Definition}
\newtheorem{example}[theorem]{Example}
\newtheorem{remark}[theorem]{Remark}
\newtheorem{conjecture}[theorem]{Conjecture}

\newcommand{\R}{\ensuremath{\mathbb{R}}}

\newcommand{\Z}{\ensuremath{\mathbb{Z}}}

\begin{document}


%
%
\renewcommand{\subjclassname}{
\textup{2020} Mathematics Subject Classification}

\date{\today}


\title[Shifting maps and cocycle invariants]{
Shifting chain maps in quandle homology and cocycle invariants
}


\author{Yu Hashimoto}
\address{Toshimagaoka-jyoshigakuen High School, 
1-25-22, Higashi-ikebukuro, Toshima-ku, Tokyo 170-0013, Japan}

\author{Kokoro Tanaka}
\address{Department of Mathematics, Tokyo Gakugei University, 
Nukuikita 4-1-1, Koganei, Tokyo 184-8501, Japan}
\email{kotanaka@u-gakugei.ac.jp}



\begin{abstract}
Quandle homology theory has been developed and 
cocycles have been used to 
define invariants of oriented classical or surface links. 
We introduce a \textit{shifting chain map} $\sigma$ on each quandle chain complex 
that lowers the dimensions by one. 
By using its pull-back 
$\sigma^\sharp$, 
each $2$-cocycle $\phi$ gives us the $3$-cocycle $\sigma^\sharp \phi$. 
%
For oriented classical links in the $3$-space, we explore relation between 
their quandle $2$-cocycle invariants associated with $\phi$
and 
their shadow $3$-cocycle invariants associated with $\sigma^\sharp \phi$. 
For oriented surface links in the $4$-space, we explore how powerful 
their quandle $3$-cocycle invariants associated with $\sigma^\sharp \phi$ are.  
Algebraic behavior of the shifting maps for low-dimensional (co)homology groups 
is also discussed.  
%
\end{abstract}

\maketitle


\section{Introduction}\label{sec:intro}

Quandles \cite{Joy,Mat} are algebraic structures whose axioms encode 
the movements of oriented classical links in $\R^3$ through their diagrams in $\R^2$. 
Later they turn out to be also useful for studying oriented surface links in $\R^4$ 
through their diagrams in $\R^3$. 
Basic link invariants derived from quandles are \textit{coloring number}s, and 
there are enhancements of coloring numbers, 
called \textit{quandle cocycle invariant}s and \textit{shadow cocycle invariant}s, 
by using quandle homology theory. 
For a quandle $X$ and an abelian group $A$, 
a (quandle) $n$-cocycle is a map from $X^n$ to $A$ satisfying some conditions; 
see Example~\ref{ex:2-cocycle} when $n=2$. 
%
%
Using a $2$-cocycle $\phi$ and a $3$-cocycle $\theta$, 
we have the following invariants of an oriented classical link $L$ in $\R^3$ 
and an oriented surface link $\mathcal{L}$ in $\R^4$:

\begin{itemize}
\item
a quandle $($$2$-$)$cocycle invariant $\Phi_\phi (L) \in \Z[A]$ 
of $L \subset \R^3$, 

\item
a quandle $($$3$-$)$cocycle invariant $\Phi_\theta (\mathcal{L}) \in \Z[A]$ 
of $\mathcal{L} \subset \R^4$, 

\item
a shadow $($$3$-$)$cocycle invariant $\Phi_\theta^x (L) \in \Z[A]$ 
of $L \subset \R^3$
for each $x \in X$, 
\end{itemize}
where $\Z[A]$ is the group ring of $A$ over $\Z$. 
Hence it is important to find $2$-cocycles and $3$-cocycles for studying 
oriented classical or surface links. 

In this paper, we introduce a \textit{shifting chain map} $\sigma$ in quandle homology theory. 
Using its dual cochain map $\sigma^\sharp$, 
the $(n+1)$-cocycle can be obtained from an $n$-cocycle for each $n \in \Z$.  
Although the precise definition 
will be given in Subsection~\ref{subsec:shift}, 
we now describe the explicit formula 
for a $2$-cocycle $\phi \colon X^2 \to A$. 
The map 
$\sigma^\sharp \phi \colon X^3 \to A$ is defined by the formula 
\begin{equation}
\sigma^\sharp \phi (x,y,z) = \phi(y,z) - \phi(x,z) + \phi(x,y) 
\label{eq:pullback}
\end{equation}
for each $x,y,z \in X$; see Example~\ref{ex:shift}.  
Given a $2$-cocycle $\phi$, 
it is natural to compare the shadow $3$-cocycle invariant 
$\Phi_{\sigma^\sharp \phi}^x(L)$ 
with the quandle $2$-cocycle invariant $\Phi_\phi (L)$  
for an oriented classical link $L$ and $x \in X$, 
and ask how powerful the quandle $3$-cocycle invariant 
$\Phi_{\sigma^\sharp \phi}(\mathcal{L})$ 
is for an oriented surface link $\mathcal{L}$. 
The following are our main results:

\begin{theorem}\label{thm:shadow-1}
Let $L$ be an oriented classical link in $\R^3$ and $\phi \colon X^2 \to A$ 
a $2$-cocycle for a quandle $X$ and an abelian group $A$. 
Then we have 
$$\Phi_{\sigma^\sharp \phi}^x (L) = \Phi_\phi (L) \in \Z[A]$$ 
for each $x \in X$.  
\end{theorem}

\begin{theorem}\label{thm:2-knot}
Let $\mathcal{L}$ be an oriented surface link in $\R^4$ and $\phi \colon X^2 \to A$ 
a $2$-cocycle for a quandle $X$ and an abelian group $A$. 
Then $\Phi_{\sigma^\sharp \phi} (\mathcal{L})$ is trivial, that is, 
$$\Phi_{\sigma^\sharp \phi} (\mathcal{L}) 
= |\mathrm{Col}_X(\mathcal{L})| \cdot 0_A \in \Z[A],$$
where 
$|\mathrm{Col}_X(\mathcal{L})|$ is the $X$-coloring number of $\mathcal{L}$ 
and 
$0_A$ is the identity element of $A$. 
\end{theorem}

Theorem~\ref{thm:shadow-1} implies that 
any quandle $2$-cocycle invariant can be considered as a shadow $3$-cocycle invariant  
through the shifting chain map $\sigma$ for classical links. 
It follows that if the quandle $2$-cocycle invariant 
associated with a $2$-cocycle $\phi$ is non-trivial for some classical link 
then the cohomology class of the $3$-cocycle $\sigma^\sharp \phi$ is non-trivial. 
On the other hand, Theorem~\ref{thm:2-knot} implies that 
the $3$-cocycle $\sigma^\sharp \phi$ gives us a trivial quandle $3$-cocycle invariant 
for surface links, 
even if the cohomology class of $\sigma^\sharp \phi$ is non-trivial.

This paper is organized as follows. 
After reviewing quandles and their homology theory, 
we define shifting chain maps, which are our main subjects, in Section~\ref{sec:shift}. 
Using a generalization of quandle homology theory reviewed in Section~\ref{sec:X-set},  
we define the quandle cocycle invariants and the shadow cocycle invariants in Section~\ref{sec:inv}. 
Key ingredients, called the fundamental quandles and the fundamental classes, for defining 
cocycle invariants are reviewed in Appendix~\ref{sec:fund}. 
Section~\ref{sec:proof} is devoted to proving Theorem~\ref{thm:shadow-1} and ~\ref{thm:2-knot}. 
Algebraic behavior of the shifting maps for low-dimensional (co)homology groups 
is studied in Section~\ref{sec:behavior}.

\section{Shifting chain maps}\label{sec:shift}
We review a quandle \cite{Joy,Mat} and its homology theory \cite{CJKLS}, 
and introduce a shifting chain map on a quandle chain complex 
that lowers the dimensions by one.

\subsection{Quandle homology}\label{subsec:quandle}
A \textit{quandle} is a set $X$ equipped with a binary operation $* \colon X \times X \to X$ 
satisfying the following three axioms for any $x,y,z \in X$: (Q1) $x * x = x$, 
(Q2) the map $s_x \colon X \to X$, 
defined by $\bullet \mapsto \bullet * x$, is bijective, 
and (Q3) $(x * y) * z = (x * z) * (y * z)$. 
Typical examples such as the dihedral quandles and the tetrahedral quandle will 
be reviewed in Section~\ref{sec:behavior}. 

Let $C_n^R(X)$ be the free abelian group generated by $X^n$ 
for $n \geq 1$ and $C_n^R(X) = 0$ for $n \leq 0$. 
Define two homomorphisms 
$\partial_n^0, \partial_n^1 \colon C_n^R(X) \to C_{n-1}^R(X)$ 
by 
\begin{align*}
\partial_n^0 (x_1, \ldots , x_n) &= 
\displaystyle\sum_{i=1}^n (-1)^i (x_1, \ldots , x_{i-1}, x_{i+1}, \ldots , x_n) \\
\partial_n^1 (x_1, \ldots , x_n) &= 
\displaystyle\sum_{i=1}^n (-1)^i (x_1*x_i, \ldots , x_{i-1}*x_i, x_{i+1}, \ldots , x_n)
\end{align*}
for $n \geq 2$ and $\partial_n^0 = \partial_n^1 = 0$ for $n \leq 1$. 
By a direct computation, we have  
\begin{equation}
\partial_{n-1}^0 \circ \partial_n^0 = 0, \quad \partial_{n-1}^1 \circ \partial_n^1 = 0
\quad \text{and} \quad 
\partial_{n-1}^1 \circ \partial_n^0 + \partial_{n-1}^0 \circ \partial_n^1 = 0.  
\label{eq:deldel}
\end{equation}
Let $\partial_n \colon C_n^R(X) \to C_{n-1}^R(X)$ be a homomorphism 
defined by $\partial_n = \partial_n^0 - \partial_n^1$. 
It follows from Equation~\eqref{eq:deldel} that $\partial_{n-1} \circ \partial_n = 0$. 
Then $C_*^R(X) = (C_n^R(X) , \partial_n)$ is a chain complex. 

Let $C_n^D(X)$ be the subgroup of $C_n^R(X)$ generated by 
\[
\{ (x_1, \ldots, x_n) \in X^n \mid x_i = x_{i+1} \ \text{for  some} \ i \}
\]
for $n \geq	 2$ and $C_n^D(X) = 0$ for $n \leq 1$. 
We have 
$\partial_n^0 (C_n^D(X)) \subset C_{n-1}^D(X)$ and 
$\partial_n^1 (C_n^D(X)) \subset C_{n-1}^D(X)$, 
hence 
$C_*^D(X) = (C_n^D(X) , \partial_n)$ is a subcomplex of $C_*^R(X)$. 
Then $C_*^Q(X) = (C_n^Q(X), \partial_n)$ is defined to be 
the quotient complex $C_*^R(X) / C_*^D(X)$ and 
called the \textit{quandle chain complex} of $X$, where 
all the induced boundary maps are again denoted by $\partial_n$'s. 
The $n$th group of cycles in $C_n^Q(X)$ is denoted by $Z_n^Q(X)$, and 
the $n$th homology group of this complex is called the $n$th 
\textit{quandle homology group} and is denoted by $H_n^Q(X)$. 

For an abelian group $A$, define the cochain complex
\[
C^*_Q(X;A) = \mathrm{Hom}_\Z(C_*^Q(X), A) 
\quad \text{and} \quad
\delta^* = \mathrm{Hom}(\partial_*, \mathrm{id})
\]
in the usual way. 
The $n$th group of cocycles in $C^n_Q(X;A)$ is denoted by $Z^n_Q(X;A)$, and 
%
the $n$th cohomology group of this complex is called the $n$th 
\textit{quandle cohomology group} and is denoted by $H^n_Q(X;A)$. 
\begin{example}\label{ex:2-cocycle}
When $C^n_R(X;A) = \mathrm{Hom}_\Z(C_n^R(X), A)$ is 
canonically identified with the set 
of all maps from $X^n$ to $A$, 
a map $\phi \colon X^2 \to A$ is a quandle $2$-cocycle if it satisfies 
\[
\phi(x,x) = 0_A \quad \text{and} \quad 
\phi(x,z) - \phi(x,y) - \phi(x*y,z) + \phi(x*z,y*z) = 0_A
\]
for any $x,y,z \in X$, where the first and second condition follow from 
\[
\phi(C^D_2(X)) = \{0_A\}
\quad  \text{and} \quad  
\phi(\partial_3(C^R_3(X)))=\{0_A\}, 
\]
respectively. 
\end{example}

\subsection{Shifting chain map}\label{subsec:shift}
Let $\sigma_n, \tilde{\sigma}_n \colon C_n^Q(X) \to C_{n-1}^Q(X)$ be two homomorphisms 
defined by 
\[ 
\sigma_n = (-1)^n \partial_n^0 \quad \text{and} \quad \tilde{\sigma}_n = (-1)^n \partial_n^1 
\]
for a quandle $X$. 
It follows from Equation~(\ref{eq:deldel}) that 
the set of maps $\sigma_n$ and that of maps $\tilde{\sigma}_n$ form 
two chain maps 
$\sigma, \tilde{\sigma} \colon C_*^Q(X) \to C_{*-1}^Q(X)$.  
These two chain maps are closely related as follows. 

\begin{prop}
The chain map $\sigma$ is chain homotopic to $\tilde{\sigma}$.
\end{prop}

\begin{proof}
Let $P_n \colon C_n^Q(X) \to C_n^Q(X)$ be a homomorphism 
defined by $P_n = (-1)^n n \cdot \mathrm{id}$. 
Direct computations show that 
\[
P_{n-1} \circ \partial_n = (-1)^{n-1} (n-1) \cdot \partial_n 
\quad \text{and} \quad 
\partial_n \circ P_{n} = (-1)^n n \cdot \partial_n , 
\]
hence  we have 
\[
P_{n-1} \circ \partial_n + \partial_n \circ P_{n} = 
(-1)^n \cdot \partial_n = \sigma_n - \tilde{\sigma}_n.
\] 
This implies that the set of maps $P_n$ is a chain homotopy between 
$\sigma$ and $\tilde{\sigma}$. 
\end{proof}

\begin{defi}
The chain map $\sigma \colon C_*^Q(X) \to C_{*-1}^Q(X)$ is called a 
\textit{shifting $($chain$)$ map} on the quandle chain complex $C_*^Q(X)$. 
\end{defi}
The shifting map $\sigma$ induces a cochain map  
$\sigma^\sharp \colon C^{*-1}_Q(X;A) \to C^{*}_Q(X;A)$ 
for an abelian group $A$, 
and then induces homomorphisms 
\[
\sigma_* \colon H^Q_*(X) \to H^Q_{*-1}(X)
\quad \text{and} \quad  
\sigma^* \colon H_Q^{*-1}(X;A) \to H_Q^*(X;A), 
\]
in the usual way. 
Behavior of these homomorphisms will be discussed in Section~\ref{sec:behavior}. 
%
\begin{example}\label{ex:shift}
A quandle $2$-cocycle $\phi \colon X^2 \to A$ 
is sent to the quandle $3$-cocycle 
$\sigma^\sharp \phi \colon X^3 \to A$ such that 
\begin{align*}
\sigma^{\sharp} \phi (x,y,z) &= \phi (\sigma_3(x,y,z)) = 
\phi \left( (-1)^3 \partial_3^0 (x,y,z) \right) \\
&= \phi \left( 
(y,z) - (x,z) + (x,y)
\right)
=\phi(y,z) - \phi(x,z) + \phi(x,y) 
\end{align*}
for each $x,y,z \in X$. 
This is nothing but Formula~\eqref{eq:pullback} shown in Section~\ref{sec:intro}. 
\end{example}

\begin{remark}
We mention here a few related topics. 
Another kind of shifting chain map lowering the dimension by one 
was discussed in \cite{CJKS-shift}. 
However, as pointed out in \cite[Remark 23]{NP}, their map is not a chain map. 
Shifting chain maps raising the dimension by one and two were defined in \cite{NP} 
for some class of quandles including dihedral quandles of odd order.  
It might be interesting to investigate the composite of our shifting chain map 
and their shifting maps. 
\end{remark}

\section{Generalized quandle homology  
}\label{sec:X-set}

We review quandle homology theory with local coefficients, 
called generalized quandle homology theory \cite{AG, CEGS}, 
which will be a bridge between quandle cocycle invariants and 
shadow cocycle invariants in Section~\ref{sec:proof}.   
Although our notational conventions for generalized quandle homology theory 
are based on \cite{Kam-book}, 
they are essentially the same, as pointed out in \cite[Remark 1 and 2]{IK}. 
We start with reviewing 
the associated group \cite{FR,Joy,Mat} of a quandle and its action. 

\subsection{Associated group of a quandle 
}
%
The \textit{associated group} $\mathrm{As}(X)$ of a quandle $X$ is defined by 
\[ 
\mathrm{As}(X):=\langle x \ (x \in X) \mid x*y = y^{-1} x y \ (x,y \in X) \rangle .
\]
A set equipped with a right action of $\mathrm{As}(X)$ is called an \textit{$X$-set}. 
The following are typical examples of $X$-sets. 

\begin{example}\label{ex:X-set}
These five $X$-sets will appear in what follows. 
%
%
\begin{enumerate}
\item
The set $X$ itself is naturally an $X$-set. 
A right action of $\mathrm{As}(X)$ on $X$ is given by 
$y \cdot x = y \ast x$ and $y \cdot x^{-1} = y \ast^{-1} x$ for $x,y \in X$, 
where $x$ is regarded as an element in $\mathrm{As}(X)$. 

\item
The set $\mathrm{As}(X)$ is naturally an $X$-set. 
A right action of $\mathrm{As}(X)$ on $\mathrm{As}(X)$ is given by 
$h \cdot g = hg$ for $g,h \in \mathrm{As}(X)$. 
%
%

\item
A singleton $Y= \{ y_0\}$ is an $X$-set. 
A right action of $\mathrm{As}(X)$ on $Y$ is given by 
$y_0 \cdot g = y_0$ for $g \in \mathrm{As}(X)$. 
%
%

\item
The set $\Z$ of all integers is an $X$-set. 
A right action of $\mathrm{As}(X)$ on $\Z$ is given by 
$y \cdot x = y + 1$ and $y \cdot x^{-1} = y -1$ for $x \in X$ and $y \in \Z$, 
where $x$ is regarded as an element in $\mathrm{As}(X)$. 

\item
For two $X$-sets $Y_1$ and $Y_2$, 
the product $Y_1 \times Y_2$ is an $X$-set. 
A right action of $\mathrm{As}(X)$ on $Y_1 \times Y_2$ is given by 
$(y_1, y_2) \cdot g = (y_1 \cdot g, y_2 \cdot g)$ for $y_1 \in Y_1$, $y_2 \in Y_2$ and 
$g \in \mathrm{As}(X)$. 
%
%
\end{enumerate}\end{example}

Let $Y,Z$ be $X$-sets. A map $p \colon Y \to Z$ is called an \textit{$X$-map} 
if it satisfies 
$$p(y \cdot g) = p(y) \cdot g$$ 
for any $y \in Y$ and $g \in \mathrm{As}(X)$.
The following are typical examples of $X$-maps. 

\begin{example}\label{ex:X-map}
These three $X$-maps will appear in what follows. 
%
%
\begin{enumerate}
\item
Let $Y$ be an $X$-set.  
For each element $y \in Y$, 
a map $s^y \colon \mathrm{As}(X) \to Y$, defined by $s^y(g) = y \cdot g$ 
for $g \in \mathrm{As}(X)$, is an $X$-map. 

\item
Let $Y_1,Y_2$ be two $X$-sets. For the product $Y_1 \times Y_2$, 
%
%
the projection $p_i \colon Y_1 \times Y_2 \to Y_i$ onto the $i$-th factor 
is an $X$-map for each $i=1,2$. 

\item
Let $Y$ be an $X$-set. For a singleton $\{y_0\}$, 
the unique map $q \colon Y \to \{y_0\}$ is an $X$-map. 
\end{enumerate}\end{example}

\subsection{Generalized quandle homology
}\label{subsec:GQ-homology}
For a quandle $X$ and an $X$-set $Y$, 
let $C_n^R(X)_Y$ be the free abelian group generated by $Y \times X^n$ 
for $n \geq 1$ and $C_0^R(X)_Y$ the free abelian group generated by $Y$. 
Put $C_n^R(X) = 0$ for $n \leq -1$. 
Define two homomorphisms 
$\partial_n^0, \partial_n^1 \colon C_n^R(X)_Y \to C_{n-1}^R(X)_Y$ 
by 
\begin{align*}
\partial_n^0 (y; x_1, \ldots , x_n) &= 
\displaystyle\sum_{i=1}^n (-1)^i (y; x_1, \ldots , x_{i-1}, x_{i+1}, \ldots , x_n) \\
\partial_n^1 (y; x_1, \ldots , x_n) &= 
\displaystyle\sum_{i=1}^n (-1)^i (y \cdot x_i; x_1*x_i, \ldots , x_{i-1}*x_i, x_{i+1}, \ldots , x_n)
\end{align*}
for $n \geq 1$ and $\partial_n^0 = \partial_n^1 = 0$ for $n \leq 0$. 
Let $\partial_n \colon C_n^R(X)_Y \to C_{n-1}^R(X)_Y$ be a homomorphism 
defined by $\partial_n = \partial_n^0 - \partial_n^1$. 
We can check that $\partial_{n-1} \circ \partial_n = 0$ in the same way,  
hence $C_*^R(X)_Y = (C_n^R(X)_Y , \partial_n)$ is a chain complex. 

Let $C_n^D(X)_Y$ be the subgroup of $C_n^R(X)_Y$ generated by 
\[
\{ (y; x_1, \ldots, x_n) \in Y \times X^n \mid x_i = x_{i+1} \ \text{for  some} \ i \}
\]
for $n \geq	 2$ and $C_n^D(X)_Y = 0$ for $n \leq 1$. 
We have $\partial_n^0 (C_n^D(X)_Y) \subset C_{n-1}^D(X)_Y$ and 
$\partial_n^1 (C_n^D(X)_Y) \subset C_{n-1}^D(X)_Y$, hence 
$C_*^D(X)_Y = (C_n^D(X)_Y , \partial_n)$ is a subcomplex of $C_*^R(X)_Y$. 
Then $C_*^Q(X)_Y = (C_n^Q(X)_Y, \partial_n)$ is defined to be 
the quotient complex $C_*^R(X)_Y / C_*^D(X)_Y$ and 
called the \textit{generalized quandle chain complex} of $X$, where 
all the induced boundary maps are again denoted by $\partial_n$'s. 
The $n$th group of cycles in $C_n^Q(X)_Y$ is denoted by $Z_n^Q(X)_Y$, and 
the $n$th homology group of this complex is called the $n$th 
\textit{generalized quandle homology group} and is denoted by $H_n^Q(X)_Y$. 
When $Y$ is a singleton $\{y_0\}$, a set of natural maps 
$C_n^Q(X)_{\{y_0\}} \to C_n^Q(X)$ defined by $(y_0; \bm{x}) \mapsto (\bm{x})$ 
becomes a chain isomorphism.
 
For an abelian group $A$, define the cochain complex
\[
C^*_Q(X;A)_Y = \mathrm{Hom}_\Z(C_*^Q(X)_Y, A) ,\quad 
\delta^* = \mathrm{Hom}(\partial_*, \mathrm{id})
\]
in the usual way. 
The $n$th group of cocycles in $C^n_Q(X;A)_Y$ is denoted by $Z^n_Q(X;A)_Y$, and 
the $n$th cohomology group of this complex is called the $n$th 
\textit{generalized quandle cohomology group} and is denoted by $H^n_Q(X;A)_Y$.

\subsection{Two types of induced chain maps
}\label{subsec:induced}
We introduce two types of induced chain maps, which will be essentially 
used to define various cocycle invariants in the next section.  

First, we define a chain map induced from a quandle homomorphism. 
Let $Q,X$ be two quandles. 
Any quandle homomorphism $f \colon Q \to X$ 
induces a group homomorphism 
$\mathrm{As}(f) \colon \mathrm{As}(Q) \to \mathrm{As}(X)$ defined by 
$\mathrm{As}(f)(x) = f(x)$ and $\mathrm{As}(f)(x^{-1}) = {f(x)}^{-1}$ for $x \in X$, 
where $x$ and $f(x)$ are regarded as an element in $\mathrm{As}(Q)$ and 
$\mathrm{As}(X)$ respectively. 
Let $f_{\sharp, n} \colon C_n^Q(Q)_{\mathrm{As}(Q)} \to C_n^Q(X)_{\mathrm{As}(X)}$ be a homomorphism defined by
by 
\[
f_{\sharp, n} (g; \bm{x}) = 
\Bigl(
\mathrm{As}(f)(g); (\overbrace{f\times \cdots \times f}^{\text{$n$ times}}) (\bm{x}) 
\Bigr) 
\]
for each $n \geq 1$ and $f_{\sharp, 0}(g) = ( \mathrm{As}(f)(g) )$, 
where $g \in \mathrm{As}(X)$ and $\bm{x} \in X^n$. 
Put $f_{\sharp, n}=0$ for $n \leq -1$. 
Then the set of maps  $f_{\sharp, n}$ 
becomes a chain map
$f_\sharp \colon C_*^Q(Q)_{\mathrm{As}(Q)} \to C_*^Q(X)_{\mathrm{As}(X)}$.

Second, we define a chain map induced from an $X$-map for a quandle $X$. 
Let $Y,Z$ be two $X$-sets, and 
$p \colon Y \to Z$ an $X$-map. 
We define a homomorphism $p_{\sharp, n} \colon C_n^Q(X)_Y \to C_n^Q(X)_Z$ 
by $p_{\sharp,n} (y; \bm{x}) = (p(y); \bm{x})$ for $n \geq 1$ and $p_{\sharp, 0}(y) = (p(y))$. 
Put $p_{\sharp,n} = 0$ for $n \leq -1$. 
Then the set of maps $p_{\sharp, n}$ 
becomes a chain map
$p_\sharp \colon C_*^Q(X)_Y \to C_*^Q(X)_Z$ . 
For each element $y \in Y$, we have two $X$-maps 
$s^y \colon \mathrm{As}(X) \to Y$ and $s^{p(y)} \colon \mathrm{As}(X) \to Z$ 
as in Example~\ref{ex:X-map} (1). 
The following is easy to prove. 

\begin{lem}\label{lem:X-map}
We have $s^{p(y)} = p \circ s^y \colon \mathrm{As}(X) \to Z$ for any $y \in Y$.  
We also have  
$s^{p(y)}_\sharp = p_\sharp \circ  s^y_\sharp \colon 
C_*^Q(X)_{\mathrm{As}(X)} \to C_*^Q(X)_Z$
for any $y \in Y$.
\end{lem}

\section{Cocycle invariants}\label{sec:inv}

We define the quandle cocycle invariant \cite{CJKLS} 
and shadow cocycle invariant \cite{CKS01}
as two types of specializations of the generalized quandle cocycle invariant \cite{CEGS}. 
We start with recalling the fundamental classes 
(of an oriented classical link \cite{CKS01,Eis} and surface link \cite{CKS01,Tan}) 
valued in the generalized homology group of the fundamental quandles 
(of the classical link \cite{Joy,Mat} and surface link \cite{Kam,PR}).  
The fundamental class is considered as a universal object for 
the generalized quandle cocycle invariant, and 
naturally derived from homological interpretation \cite{CKS03,FRS} 
of the invariant. 

\subsection{Fundamental quandle and fundamental class}
%
To each diagram $D$ of an oriented classical link (or surface link), 
we can associate a quandle and a homology class as invariants of the link; 
the \textit{fundamental quandle} $Q_D$ and the 
\textit{fundamental class} $[D]$, where  $[D]$ takes its value 
in the $2$nd (or $3$rd) generalized quandle homology group 
$H_*^Q(Q_D)_{\mathrm{As}(Q_D)}$.  
%
Since we do not use the precise constructions of both invariants in the rest of the paper,
we postpone their definition to Appendix~\ref{sec:fund} 
and here introduce their important properties instead. 
For another diagram $D'$, 
if the oriented classical or surface link represented by $D'$ 
is equivalent to that represented by $D$, 
then there exists a quandle isomorphism from $Q_D$ to $Q_{D'}$ such that 
the induced group isomorphism on the 
homology groups sends $[D]$ to $[D']$; see Theorem~\ref{thm:1-inv} and~\ref{thm:2-inv}.  

Let $X$ be a finite quandle. 
For a diagram $D$ of an oriented classical link $L$ or an oriented surface link $\mathcal{L}$, 
let $\mathrm{Col}_X (D)$ be the set of all quandle homomorphisms from $Q_D$ to $X$.    
It follows from the aforementioned properties of $Q_D$ that  
its cardinality $| \mathrm{Col}_X (D) |$ does not depend on the choice of the diagram 
$D$ for the link. 
Hence, when $D$ represents $L$ (or $\mathcal{L}$),
it is also denoted by 
$| \mathrm{Col}_X (L) |$ (or $| \mathrm{Col}_X (\mathcal{L}) |$) 
and called the \textit{$X$-coloring number} of $L$ (or $\mathcal{L}$). 

\subsection{Generalized quandle cocycle invariant
}
Let $X$ be a finite quandle, $A$ an abelian group, $Y$ an $X$-set and $y$ an element in $Y$.  
Let $D$ be a diagram of an oriented classical link $L$ or an oriented surface link $\mathcal{L}$.  
When $D$ represents the classical link (or surface link), 
for a given $2$-cocycle (or $3$-cocycle) $\lambda \in Z^*_Q(X;A)_Y$, 
we define a \textit{generalized quandle cocycle invariant} $\Psi_\lambda^y (D)$ 
by  
\[
\Psi_\lambda^y (D) = \displaystyle\sum_{c \colon Q_D \to X} 
\left\langle  s^y_* \circ c_* [D],\ [\lambda]  \right\rangle \ \in \Z[A] , 
\]
where $c_* \colon H_*^Q(Q_D)_{\mathrm{As}( Q_D )} \to H_*^Q(X)_{\mathrm{As} (X)}$ 
is the induced homomorphism from a quandle homomorphism $c \colon Q_D \to X$ 
as in the first half of Subsection~\ref{subsec:induced}, 
the map $s^y_* \colon H_*^Q(X)_{\mathrm{As} (X)} \to  H_*^Q(X)_Y$ is 
the induced homomorphism from an $X$-map $s^y \colon \mathrm{As}(X) \to Y$
as in the second half of Subsection~\ref{subsec:induced},  
the element $[\lambda]$ is a cohomology class of $\lambda$, and 
\[
\langle \ , \ \rangle \colon H_*^Q(X)_Y \otimes H^*_Q(X;A)_Y \to A
\]
is a Kronecker product. 
We note that the above summation is finite, since the cardinality of X is finite. 
By definition, the invariant $\Psi_\lambda^y (D)$  
depends only on the cohomology class $[\lambda]$ of $\lambda$. 
Thus, when $\lambda$ is a coboudary, the invariant $\Psi_\lambda^y (D)$ is trivial, 
that is, 
$$\Psi_\lambda^y (D) = |\mathrm{Col}_X (D)| \cdot 0_A \in \Z[A],$$ 
where $0_A$ is the identity element of $A$. 
It follows from Theorem~\ref{thm:1-inv} and~\ref{thm:2-inv} that 
the invariant $\Psi_\lambda^y (D)$ does not depend on the choice of the diagram $D$ for the link. 
Hence, when $D$ represents $L$ (or $\mathcal{L}$), 
it is also denoted by $\Psi_\lambda^y (L)$ (or $\Psi_\lambda^y (\mathcal{L})$). 
%

\subsection{Quandle cocycle invariant}
Let $X$ be a finite quandle and $A$ an abelian group. 
Let $L$ be an oriented classical link and $\mathcal{L}$ an oriented surface link.   
For given $2$-cocycle $\phi \in Z^2_Q(X;A)$ and $3$-cocycle  $\theta \in Z^3_Q(X;A)$, 
we define \textit{quandle cocycle invariant}s $\Phi_\phi (L)$ and $\Phi_\theta (\mathcal{L})$ 
by 
\[
\Phi_\phi (L) = \Psi_\phi^{y_0} (L) \in \Z[A] 
\quad \text{and} \quad 
\Phi_\theta (\mathcal{L}) = \Psi_\theta^{y_0} (\mathcal{L}) \in \Z[A] ,
\]
where $\phi$ and $\theta$ are regarded as the elements 
in $Z^*_Q(X;A)_{\{ y_0 \}}$ for a singleton $\{ y_0\}$
by naturally identifying $C_Q^*(X;A)_{\{y_0\}}$ with $C_Q^*(X;A)$; 
see Subsection~\ref{subsec:GQ-homology}.

\subsection{Shadow cocycle invariant}\label{subsec:shadow}
For a quandle $X$,  
we define a homomorphism $\iota_n \colon C_n^Q(X)_X \to C_{n+1}^Q(X)$ 
by 
$\iota_n(x_0; \bm{x}) = (-1)^n (x_0, \bm{x})$
for $n \geq 1$ and $\iota_0 (x_0) = (x_0)$. Put $\iota_n = 0$ for $n \leq -1$. 
It follows from Equation~\eqref{eq:deldel} in Subsection~\ref{subsec:quandle} that 
the set of maps $\iota_n$ forms a chain map $\iota \colon C_*^Q(X)_X \to C_{*+1}^Q(X)$.   
Let $A$ be an abelian group and fix an element $x$ in $X$. 
Let $L$ be an oriented classical link and $\mathcal{L}$ an oriented surface link.   
When $X$ is finite, 
for given $3$-cocycle $\theta \in Z^3_Q(X;A)$ and $4$-cocycle  $\psi \in Z^4_Q(X;A)$, 
we define \textit{shadow cocycle invariant}s  $\Phi_\theta^x (L)$ and $\Phi_\psi^x (\mathcal{L})$  
by 
\[
\Phi_\theta^x (L) = \Psi_{\iota^\sharp \theta}^x (L)  \in \Z[A] 
\quad \text{and} \quad 
\Phi_\psi^x (\mathcal{L}) = \Psi_{\iota^\sharp \psi}^x (\mathcal{L})  \in \Z[A] ,
\]
where $\iota^\sharp \colon Z^{* +1}_Q(X;A) \to Z^*_Q(X;A)_X$ is 
the pull-back induced by the chain map $\iota$.

\section{Proof}\label{sec:proof}

We prove our main theorems (Theorem~\ref{thm:shadow-1} and ~\ref{thm:2-knot}),  
using a duality (Proposition~\ref{prop:dual}) for generalized quandle cocycle invariants.  

\subsection{Duality for invariants}
For a finite quandle $X$, 
let $Y,Z$ be two $X$-sets and 
$p \colon Y \to Z$ an $X$-map. 
Let $D$ be a diagram of an oriented classical or surface link.  
When $D$ represents the classical link (or surface link), 
%
let $\lambda$ be a $2$-cocycle (or $3$-cocycle) in $ Z^*_Q(X;A)_Z$. 

\begin{prop}\label{prop:dual}
We have 
\[
\Psi_{\lambda}^{p(y)} (D) = \Psi_{p^\sharp \lambda}^y (D) \in \Z[A]
\]
for any $y \in Y$, 
where $p^\sharp \colon Z^*_Q(X;A)_Z \to Z^*_Q(X;A)_Y$ is 
the pull-back induced by $p$. 
\end{prop}

\begin{proof}
By a direct calculation, we have 

\begin{align*}
\Psi_{\lambda}^{p(y)} (D) 
& = \displaystyle\sum_c \langle s^{p(y)}_* \circ c_* [D], [\lambda] \rangle 
= \displaystyle\sum_c \langle (p_* \circ s^{y}_*) \circ c_* [D], [\lambda] \rangle \\
& = \displaystyle\sum_c \langle p_* \circ (s^{y}_* \circ c_*) [D], [\lambda] \rangle 
= \displaystyle\sum_c \langle s^{y}_* \circ c_* [D], p^* [\lambda] \rangle 
= \Psi_{p^\sharp \lambda}^y (D) , 
\end{align*}
where the fourth equality follows from the usual duality of the Kronecker product 
and the second equality follows from Lemma~\ref{lem:X-map}.
%
\end{proof}

\subsection{Proof of Theorem~\ref{thm:shadow-1}}

Let $X$ be a quandle. 
For the two $X$-sets, $\Z$ and $X$, 
we have an $X$-set $\Z \times X$; see Example~\ref{ex:X-set}.
Two maps, the projection $p \colon \Z \times X \to X$  
and the unique map $q \colon X \to \{y_0\}$, 
are $X$-maps; see Example~\ref{ex:X-map}.
Even if the composite
$$\sigma \circ \iota \colon  C_*^Q(X)_X \to  C_*^Q(X)$$
of the shifting chain map $\sigma$ and the chain map $\iota$
is not chain homotopic to the chain map 
$$q_\sharp \colon C_*^Q(X)_X \to C_*^Q(X)_{\{y_0\}} 
\bigl( = C_*^Q(X) \bigr),$$ 
we can show the following by composing the chain map 
$$p_\sharp \colon C_*^Q(X)_{\Z \times X} \to C_*^Q(X)_X . $$ 

\begin{prop}\label{prop:shadow-1}
The chain map $(\sigma \circ \iota) \circ p_\sharp$ is chain homotopic to 
$q_\sharp \circ p_\sharp$.
\end{prop}

\begin{proof}
Let $P_n \colon C_n^Q(X)_{\Z \times X} \to C_{n+1}^Q(X)$ 
be a homomorphism defined by 
$$P_n (a,x_0; \bm{x}) = a \cdot (x_0, \bm{x}) $$
for each generator $(a, x_0; \bm{x}) \in C_n^Q(X)_{\Z \times X}$. 
Direct computations show that 
\begin{align*}
P_{n-1} \circ \partial_n (a,x_0; \bm{x}) 
& = -a \cdot \partial_{n+1}(x_0, \bm{x}) + \partial_{n+1}^1 (x_0, \bm{x}) + (\bm{x}), \\
\partial_{n+1} \circ P_n (a,x_0; \bm{x})
& = a \cdot \partial_{n+1}(x_0, \bm{x}), \\
%
q_{\sharp n} \circ p_{\sharp n} (a,x_0; \bm{x}) 
& = q_{\sharp n}(x_0; \bm{x}) = (\bm{x}), 
\end{align*}
and 
$$(\tilde{\sigma} \circ \iota)_n \circ p_{\sharp n} (a,x_0; \bm{x}) 
= \tilde{\sigma}_{n+1} \circ \iota_n (x_0; \bm{x}) = - \partial_{n+1}^1 (x_0, \bm{x}).$$
Hence we have 
$$P_{n-1} \circ \partial_n + \partial_{n+1} \circ P_n
= q_{\sharp n} \circ p_{\sharp n} - (\tilde{\sigma} \circ \iota)_n \circ p_{\sharp n}, $$
and this implies that the set of maps $P_n$ is a chain homotopy between 
$(\sigma \circ \iota) \circ p_\sharp$ and $q_\sharp \circ p_\sharp$. 
\end{proof}

\begin{proof}[Proof of Theorem~\ref{thm:shadow-1}]
By a direct calculation, for any $m \in \Z$, we have 
\begin{align*}
\Phi_{\sigma^\sharp \phi}^x (L) 
& = \Psi_{\iota^\sharp (\sigma^\sharp \phi)}^x (L) 
= \Psi_{(\sigma \circ \iota)^\sharp \phi}^{p(m,x)} (L)  
= \Psi_{p^\sharp ( (\sigma \circ \iota)^\sharp \phi )}^{(m,x)} (L) 
= \Psi_{p^\sharp ( q^\sharp \phi )}^{(m,x)} (L) \\  
& = \Psi_{q^\sharp \phi}^{p(m,x)} (L) 
= \Psi_{q^\sharp \phi}^x (L) 
= \Psi_{\phi}^{q(x)} (L)
= \Psi_{\phi}^{y_0} (L)
= \Phi_{\phi} (L) ,
\end{align*}
where the third and seventh equalities follow from the duality (Proposition~\ref{prop:dual})
and the fourth equality follows from Proposition~\ref{prop:shadow-1}.  
\end{proof}

By the same argument, we can show 
the one-dimensional higher version of Theorem~\ref{thm:shadow-1}. 
Here we just state the result.  
\begin{theorem}\label{thm:shadow-2}
Let $\mathcal{L}$ be an oriented surface link in $\R^4$ and $\theta \colon X^3 \to A$ 
a $3$-cocycle for a quandle $X$ and an abelian group $A$. 
Then we have 
$$\Phi_{\sigma^\sharp \theta}^x (\mathcal{L}) 
= \Phi_\theta (\mathcal{L}) \in \Z[A]$$
for each $x \in X$.  
\end{theorem}

\subsection{Proof of Theorem~\ref{thm:2-knot}}

Let $X$ be a quandle. For the $X$-set $\Z$, 
the unique map $q \colon \Z \to \{y_0\}$ 
is an $X$-map. 
Even if the shifting chain map 
$$\sigma \colon C_*^Q(X) \to C_{*-1}^Q(X)$$
itself is not null-homotopic, 
we can show the following by composing the chain map 
$$q_\sharp \colon C_*^Q(X)_\Z \to C_*^Q(X)_{\{y_0\}} \bigl( = C_*^Q(X) \bigr).$$

\begin{prop}\label{prop:2-knot}
The chain map $\sigma \circ q_\sharp$ is null-homotopic. 
\end{prop}

\begin{proof}
Let $P_n \colon C_n^Q(X)_\Z \to C_n^Q(X)$ be a homomorphism defined by 
$$P_n(a; \bm{x}) = (-1)^n a \cdot (\bm{x}) .$$
for each generator $(a; \bm{x}) \in C_n^Q(X)_{\Z}$.
Direct computations show that 
\begin{align*}
P_{n-1} \circ \partial_n (a; \bm{x}) 
& = (-1)^{n-1} a \cdot \partial_n(\bm{x}) + (-1)^{n} \partial_n^1(\bm{x}), \\
\partial_n \circ P_n (a; \bm{x}) 
& = (-1)^n a \cdot \partial_n (\bm{x}), \\
\end{align*}
and 
$$\tilde{\sigma}_n \circ q_{\sharp n} (a; \bm{x}) 
 = \tilde{\sigma}_n (\bm{x}) = (-1)^n \partial_n^1(\bm{x}) .$$
Hence we have 
$$P_{n-1} \circ \partial_n + \partial_n \circ P_n = \tilde{\sigma}_n \circ q_{\sharp n}, $$
and this implies that the set of maps $P_n$ is a chain homotopy between 
$\sigma \circ q_\sharp$ and the $0$-map. 
\end{proof}

\begin{proof}[Proof of Theorem~\ref{thm:2-knot}]
By a direct calculation, for any $m \in \Z$, we have 
\[
\Phi_{\sigma^\sharp \phi} (\mathcal{L}) 
= \Psi_{\sigma^\sharp \phi}^{y_0} (\mathcal{L}) 
= \Psi_{\sigma^\sharp \phi}^{q(m)} (\mathcal{L}) 
= \Psi_{q^\sharp (\sigma^\sharp \phi)}^m (\mathcal{L}) 
= \Psi_{\theta_{\mathrm{triv}}}^m (\mathcal{L}) 
= |\mathrm{Col}_X(\mathcal{L})| \cdot 0_A ,
\]
where $\theta_{\mathrm{triv}}$ is the trivial $3$-cocycle in $Z_Q^3(X;A)$, 
the third equality follows from the duality (Proposition~\ref{prop:dual}) and 
the fourth equality follows from Proposition~\ref{prop:2-knot}.  
\end{proof}

By the same argument, we can show 
the one-dimensional lower version of Theorem~\ref{thm:2-knot}. 
Here we just state the result.  
\begin{theorem}\label{thm:1-knot}
Let $L$ be an oriented classical link in $\R^3$ and $\kappa \colon X \to A$ 
a $1$-cocycle for a quandle $X$ and an abelian group $A$. 
Then $\Phi_{\sigma^\sharp \kappa} (L)$ is trivial, that is, 
$$\Phi_{\sigma^\sharp \kappa} (L) = |\mathrm{Col}_X(L)| \cdot 0_A \in \Z[A],$$
where 
$|\mathrm{Col}_X({L})|$ is the $X$-coloring number of ${L}$ 
and 
$0_A$ is the identity element of $A$. 
\end{theorem}

\section{Behavior of the shifting maps for (co)homology groups}\label{sec:behavior}

We study behavior of the shifting maps for low-dimensional (co)homology groups. 
First, we discuss behavior between $1$st and $2$nd (co)homology groups 
for connected quandles and the trivial quandle of order $2$. 
Second, we discuss behavior 
between low-dimensional (co)homology groups greater than one dimension 
for specific connected quandles;   
the dihedral quandles of odd prime orders and the tetrahedral quandle (of order $4$).

\subsection{Behavior between $1$st and $2$nd (co)homology groups}
Given a quandle $X$, 
it follows from (Q2) and (Q3) of the quandle axioms that the map $s_x \colon X \to X$ is 
an automorphism, called an \textit{inner automorphism}, of $X$ for each $x \in X$. 
The quandle $X$ is said to be \textit{connected} if the group generated by 
all inner automorphisms acts transitively on $X$. 
Then, for any $x \in X$, 
the homology class $[(x)]$ of $(x) \in Z^Q_1(X) = C^Q_1(X)$  
generates  $H^Q_1(X) \cong \Z$ 
and hence we have $H_Q^1(X;A) \cong \mathrm{Hom}(H^Q_1(X),A) \cong A$
for 
any abelian group $A$; see \cite{CJKS-Betti}.  
%
Behavior of the shifting map between 
$1$st and $2$nd (co)homology groups for connected quandles 
is the following.
\begin{theorem}
The shifting maps 
\[\sigma_* \colon H^Q_2(X) \to H^Q_1(X) 
\quad \text{and} \quad 
\sigma^* \colon H_Q^1(X;A) \to H_Q^2(X;A)\] 
are the $0$-maps for any connected quandle $X$ and abelian group $A$. 
\end{theorem}

\begin{proof}
For any generator $(x,y)$ in $C^Q_2(X)$, 
we have 
\[
\sigma_2(x,y) = - (y) + (x) 
\quad \text{and} \quad 
\partial_2(x,y) = + (x) - (x*y) 
\]
in $C^Q_1(X)$ at the chain level.   
Since the connectedness of $X$ implies $[(x)] = [(y)] \in H^Q_1(X)$, 
we have $[\sigma_2(x,y)]=0\in H^Q_1(X)$, and hence $\sigma_*$ is the $0$-map. 
Since the connectedness of $X$ also implies that 
any $1$-cocycle  $\kappa \colon X \to A$ is a constant map, 
we have 
$\sigma^* \kappa (x,y) = \kappa(\sigma_2(x,y)) = -\kappa(y) + \kappa(x) = 0$, and hence 
$\sigma^*$ is the $0$-map. 
\end{proof}

Behavior 
for disconnected quandles is totally different.  
The \textit{trivial quandle} $T_2$ of order $2$ is defined to be a set $\{0,1\}$ with 
the binary operation $x*y=x$ for each $x,y \in \{0,1\}$. 
This is the most simplest disconnected quandle, and 
the set of $[(0)]$ and $[(1)]$ generates $H^Q_1(T_2) \cong  \Z^2$
and that 
the set of $[(0,1)]$ and $[(1,0)]$ generates $H^Q_2(T_2) \cong  \Z^2$; see  \cite{CJKS-Betti}. 
Then we have 
\[ H_Q^1(T_2;A) \cong \mathrm{Hom}(H^Q_1(T_2);A) \cong A^2
\ \text{and} \
H_Q^2(T_2;A) \cong \mathrm{Hom}(H^Q_2(T_2),A) \cong A^2 \] 
for any abelian group $A$. 
Since $\sigma_2(0,1) = - (1) + (0)$ and $\sigma_2(1,0) = - (0) + (1)$ at the chain level, 
we have the following. 

\begin{prop}
The shifting map
$\sigma_* \colon H^Q_2(T_2) \to H^Q_1(T_2)$ 
is a homorphism $\Z^2 \to \Z^2$ such that $(1,0) \mapsto (1,-1)$ and 
$(0,1) \mapsto (-1,1)$. 
Dually, for any abelian group $A$, 
the shifting map $\sigma^* \colon H_Q^1(X;A) \to H_Q^2(X;A)$ 
is a homorphism $A^2 \to A^2$ such that $(a,0) \mapsto (a,-a)$ and 
$(0,a) \mapsto (-a,a)$ for each $a \in A$. 
\end{prop}

\subsection{Dihedral quandle of odd prime order}
The \textit{dihedral quandle} $R_n$ of order $n$ is defined to be $\Z_n$ 
with the binary operation $x * y = 2x - y$ for each $x,y \in R_n$, 
where $\Z_n$ denotes the cyclic group of order $n$. 
The quandle $R_n$ is connected if and only if the number $n$ is odd. 
Let $p$ be an odd prime throughout in this subsection. 
It is known that 
\[
H^Q_2(R_p) \cong 0,\ H^Q_3(R_p) \cong \Z_p,\ H^Q_4(R_p) \cong \Z_p 
\]
and 
\[
H_Q^2(R_p;\Z_p) \cong 0,\ 
H_Q^3(R_p;\Z_p) \cong \Z_p,\ H_Q^4(R_p;\Z_p) \cong \Z_p^2
\]
for any odd prime $p$; see \cite{Moc03} for the $2$nd 
(co)homology group and $3$rd cohomology group, 
\cite{NP} for the $3$rd homology group, and 
\cite{Cla,Nos13-TAMS} for the $4$th (co)homology group.  
The first non-trivial shifting maps for (co)homology groups of $R_p$ might be 
\[
\sigma_* \colon H^Q_4(R_p) \to H^Q_3(R_p)
\quad \text{and} \quad  
\sigma^* \colon H_Q^3(R_p;\Z_p) \to H_Q^4(R_p;\Z_p), 
\]
which we are going to study. 

Let $z, w$ be cycles in $Z^Q_*(R_p)$ defined by 
\begin{align*}
%
z &= \sum_{i=1}^{p-2} \Bigl( (0,i,i+1) - (i,i+1,0) + (0,i+1,i) - (i+1,i,0) \Bigr) & \in Z^Q_3(R_p), \\
%
w &= -\sum_{i=1}^{p-2} \Bigl( (0,i,i+1,0) + (0,i+1,i,0) \Bigr) & \in Z^Q_4(R_p), 
\end{align*}
where these two cycles are obtained from 
a diagram of the $2$-twist spun $(2,p)$-torus knot with a (shadow) $R_p$-coloring. 
Let ${\chi} \colon R_p^2 \to \Z_p$ be a $2$-cochain in $C_Q^2(R_p;\Z_p)$ defined by 
\[
{\chi}(x,y) := \dfrac{(2y-x)^p + x^p - 2y^p}{p} \ \Biggl( \equiv
\sum_{i=1}^{p-1} i^{-1} (-x)^i (2y)^{p-i} \pmod p \Biggr) .
\]
Using the $2$-cochain ${\chi}$, we define 
a $3$-cocycle $\theta \in Z_Q^3(R_p;\Z_p)$, known as Mochizuki's $3$-cocycle \cite{Moc03},  
and $4$-cocycles $\psi_0, \psi_1 \in Z_Q^4(R_p;\Z_p)$ 
given in \cite[Example 5.9]{Nos13-TAMS}\footnote{
There is a minor typo: the term \lq\lq $(2z-w)$\rq\rq\ 
should be replaced by \lq\lq $(2w-z)$\rq\rq\ in his formula of $\psi_1$.
We note that the symbols $\psi_{4,0}$ and $\psi_{4,1}$ are used in \cite{Nos13-TAMS} 
instead of $\psi_0$ and $\psi_1$. 
}\ 
by 
\begin{align*}
\theta(x,y,z) &= (x-y) \cdot {\chi}(y,z), \\
\psi_0(x,y,z,w) &= - \theta(x-w,y-w,z-w), \\ 
\psi_1(x,y,z,w) &= {\chi}(z-x, y-x) \cdot {\chi}(z,w), 
\end{align*}
where this explicit expression for $\theta$ was that simplified in \cite{AS}.  
Then it is proved in \cite{Moc03} that 
the cohomology class $[\theta]$ generates $H_Q^3(R_p;\Z_p) \cong \Z_p$, 
and in \cite{Nos13-TAMS} that 
the set of the cohomology classes $[\psi_0]$ and $[\psi_1]$ 
generates $H_Q^4(R_p;\Z_p) \cong \Z_p^2$. 
By a direct computation, we have $\theta(z) = -2$, $\psi_0(w) = -\theta(-z) = -2$, 
and hence the following.

\begin{prop}
The homology class $[z]$ is a generator of $H^Q_3(R_p) \cong \Z_p$ and 
$[w]$ is that of $H^Q_4(R_p) \cong \Z_p$. 
\end{prop}

\begin{remark}
Let $z_0 = \sum_{i=1}^{p-1} (0,i,i+1)$ be a $3$-cycle in $C^Q_3(R_p)$, 
where this cycle is obtained from a diagram of the $(2,p)$-torus knot with a shadow $R_p$-coloring. 
It is known that $[z_0]$ is also a generator of $H^Q_3(R_p) \cong \Z_p$ and $[z] = 2[z_0]$, 
since $\theta(z_0)=-1$.
\end{remark}

\begin{theorem}
The map $\sigma_* \colon H^Q_4(R_p) \to H^Q_3(R_p)$ 
is an isomorphism $\Z_p \to \Z_p$ such that $1 \mapsto -1$. 
\end{theorem}

\begin{proof}
By a direct computation, we have $\sigma(w) = - z$ 
at the chain level. 
\end{proof}

\begin{prop}\label{prop:R_p}
The map $\sigma^* \colon H_Q^3(R_p;\Z_p) \to H_Q^4(R_p;\Z_p)$ 
is an injective homomorphism $\Z_p \to \Z_p^2$ 
such that $1$ $\mapsto$ $(-1,n)$ for some $n \in \Z_p$.
\end{prop}

\begin{proof}
By a direct computation, 
we have $\sigma^\sharp(\theta) (w) = \theta(\sigma(w)) = \theta(-z) = 2$. 
Using $\chi (x,0) = 0 \in \Z_p$ for any $x \in R_p$, we have $\psi_1(w) = 0$. 
Combined these with $\psi_0(w) = -2$, 
we have $\sigma^* ([\theta]) = -[\psi_0] + n [\psi_1]$ for some $n \in \Z_p$. 
\end{proof}

Katsumi Ishikawa kindly computed $\sigma^*$ 
by his computer and verified that $n = 0 \in \Z_p$ for $p = 3,5,7,11,13$.  
Thus we propose the following (cf. Theorem~\ref{thm:S_4}). 

\begin{conjecture}
The map $\sigma^* \colon H_Q^3(R_p;\Z_p) \to H_Q^4(R_p;\Z_p)$ 
is an injective homomorphism $\Z_p \to \Z_p^2$ 
such that $1$ $\mapsto$ $(-1,0)$.
\end{conjecture}


\subsection{Tetrahedral quandle}
The \textit{tetrahedral quandle} $S_4$ is defined to be 
the set of vertices of a regular tetrahedron, denoted by  $\{0,1,2,3\}$, 
with the binary operation 
\[\begin{array}{ccccccccc}
0*0 &=& 1*2 &=& 2*3 &=& 3*1 &=& 0,  \\
0*3 &=& 1*1 &=& 2*0 &=& 3*2 &=& 1,  \\
0*1 &=& 1*3 &=& 2*2 &=& 3*0 &=& 2,  \\
0*2 &=& 1*0 &=& 2*1 &=& 3*3 &=& 3. 
\end{array}\]
The inner automorphism $s_x \colon S_4 \to S_4$ for $x \in S_4$ 
can be regarded as a counterclockwise rotation of the tetrahedron by 
the angle $2 \pi/3$ looking at the bottom face from the vertex $x$. 
This quandle $S_4$ is known to be connected. 
%
It is known in \cite{CJKLS,LN} that 
\[
H^Q_2(S_4) \cong \Z_2,\ H^Q_3(S_4) \cong \Z_2 \oplus \Z_4 
\]
and 
\[
H_Q^2(S_4;\Z_4) \cong \Z_2,\  
H_Q^3(S_4;\Z_4) \cong \Z_2^2 \oplus \Z_4 ( = (\Z_2 \oplus \Z_4) \oplus \Z_2 ),
\]
where $H_Q^3(S_4;\Z_4) \cong \mathrm{Hom}(H^Q_3(S_4),\Z_4) \oplus \mathrm{Ext}(H^Q_2(S_4),\Z_4) 
\cong (\Z_2 \oplus \Z_4) \oplus \Z_2$. 
The first non-trivial shifting maps for (co)homology groups of $S_4$ might be 
\[
\sigma_* \colon H^Q_3(S_4) \to H^Q_2(S_4)
\quad \text{and} \quad  
\sigma^* \colon H_Q^2(S_4;\Z_4) \to H_Q^3(S_4;\Z_4), 
\]
which we are going to study.

Let $z_1,z_2,w_1,w_2$ be cycles in $Z^Q_*(S_4)$ defined by 
\begin{align*}
z_1 &= +(0,3) + (3,1) + (1,0) & \in Z^Q_2(S_4), \\
z_2 &= +(0,3) - (0,1) + (3,0) - (3,2) & \in Z^Q_2(S_4), \\
w_1 &= +(0,3,1) + (0,1,0) & \in Z^Q_3(S_4), \\
w_2 &= - (1,0,1) + (0,3,0) - (0,3,2) & \in Z^Q_3(S_4), 
\end{align*}
where cycles $z_1,w_1$ and $z_2,w_2$ are obtained from 
diagrams of the trefoil knot and the figure-eight knot with (shadow) $S_4$-colorings, 
respectively. 
Given an element $\bm{x} \in S_4^n$, let $\chi_{\bm{x}} \colon S_4^n \to \Z_4$ be 
an $n$-cochain in $C_Q^n(S_4;\Z_4)$ 
defined by 
\[\chi_{\bm{x}} (\bm{y}) = \begin{cases}1 & (\bm{y}=\bm{x}) \\ 0 & (\bm{y} \neq \bm{x})\end{cases}\]  
for each $\bm{y} \in S_4^n$. 
Using the above $n$-cochains, we define  
a $2$-cocycle ${\phi} \in Z_Q^2(S_4;\Z_4)$ and 
$3$-cocycles  ${\eta}_1, \eta_2, \eta_{11} \in Z_Q^3(S_4;\Z_4)$
by 
\begin{align*}
{\phi} 
&=
2(+\chi_{(0,1)}+\chi_{(0,2)}+\chi_{(1,0)}+\chi_{(1,2)}
+\chi_{(2,0)}+\chi_{(2,1)}), \\
{\eta}_{1} 
&=
2( +\chi_{(0,1,0)}+\chi_{(0,2,1)}+\chi_{(0,2,3)}+\chi_{(0,3,0)}
+\chi_{(0,3,1)}+\chi_{(0,3,2)}+\chi_{(1,0,1)} \\ 
&\phantom{=}
+\chi_{(1,0,3)}
+\chi_{(1,2,0)}+\chi_{(1,3,1)}+\chi_{(2,0,3)}+\chi_{(2,1,0)}
+\chi_{(2,1,3)}+\chi_{(2,3,2)} ), \\
\eta_{2}&=
+\chi_{(0,1,2)}-\chi_{(0,1,3)}-\chi_{(0,2,1)}+\chi_{(0,3,0)}
+\chi_{(0,3,1)}-\chi_{(0,3,2)}+2\chi_{(1,0,1)} \\ 
&\phantom{=}
+\chi_{(1,0,2)} 
+\chi_{(1,0,3)}-\chi_{(1,2,0)} 
+\chi_{(1,3,2)}+\chi_{(2,0,1)}
+\chi_{(2,0,2)}+\chi_{(2,0,3)} \\ 
&\phantom{=}
+\chi_{(2,1,3)}+\chi_{(3,0,1)}
+\chi_{(3,0,2)}+\chi_{(3,0,3)}+\chi_{(3,1,3)}, \\	
\eta_{11}&=
-\chi_{(0,1,0)}-\chi_{(0,1,3)}+\chi_{(0,3,1)}
+\chi_{(0,3,2)}-\chi_{(1,0,1)}-\chi_{(1,0,2)}-\chi_{(1,0,3)} \\ 
&\phantom{=}
+\chi_{(1,2,0)}-\chi_{(1,2,1)}+\chi_{(1,3,0)}+\chi_{(1,3,1)}
+\chi_{(1,3,2)}+\chi_{(2,0,3)}-\chi_{(2,1,0)} \\ 
&\phantom{=}
-\chi_{(3,0,2)}+\chi_{(3,2,3)},  
\end{align*}
where the $2$-cocycle and $3$-cocycles were given in the proofs 
of Lemma 6.8 and 6.12, respectively, in the $2$nd version of pre-publication 
paper\footnote{
Not only the proofs, but also the lemmas themselves have been  
deleted in the published version. 
We note that our $\phi \colon S_4^2 \to \Z_4$ is twice their 
$\phi \colon S_4^2 \to \Z_2$, and 
that our $\eta_1 \colon S_4^3 \to \Z_4$ is twice their 
$\eta_1 \colon S_4^3 \to \Z_2$. 
}\ 
of \cite{CJKLS} on the arXiv. 
The explicit expressions for these cocycles 
are described in \cite[Example 8.11]{CJKLS} 
and \cite[p.64]{CJKS-Adv}.   
%
Then it is known that 
the cohomology class $[{\phi}]$ generates $H_Q^2(S_4;\Z_4) \cong \Z_2$, and that 
the (ordered) set of the cohomology classes $[{\eta}_1]$, $[\eta_2]$ and $[\eta_{11}]$ 
generates  
$H_Q^3(S_4;\Z_4) \cong \Z_2^2 \oplus \Z_4$ ($= (\Z_2 \oplus \Z_4) \oplus \Z_2$ in this order); 
see \cite{CJKLS, CJKS-Adv}.  

\begin{prop}
We have $[z_1] = [z_2]$ and this generates $H^Q_2(S_4) \cong \Z_2$. 
\end{prop}

\begin{proof}
Since ${\phi}(z_1) = {\phi}(z_2) = 2 \in \Z_4$,  $[z_1]$ and $[z_2]$ 
are non-zero in $H^Q_2(S_4) \cong \Z_2$. 
\end{proof}

\begin{prop}
A set of homology classes $[w_1]$ and $[w_2]$ 
generates $H^Q_3(S_4) \cong \Z_2 \oplus \Z_4$. 
Moreover, $[w_1]$ and $[w_2]$ have order $4$ and $2$, respectively. 
\end{prop}

\begin{proof}
It suffices to prove that $[w_1]$ has order $4$, 
$[w_2]$ has order $2$, and $[w_2] \neq 2 [w_1]$. 
It follows from $\eta_2(w_1) = 1 \in \Z_4$ that 
$[w_1]$ has order $4$ and $\eta_2$ sends any element of order $4$ to 
a non-zero element in $\Z_4$. 
%
%
Since $\eta_2(w_2) = 0 \in \Z_4$, 
$[w_2]$ does not have order $4$ 
and $[w_2] \neq 2 [w_1]$. 
It follows from ${\eta}_1(w_2) = 2 \in \Z_4$ that 
$[w_2]$ has order $2$. 
%
\end{proof}

\begin{remark}
Computations similar to ours in the above proof have been appeared 
in the proof of \cite[Proposition 4.5]{Nos11}, 
where he attempted to compute $\pi_2(BS_4)$ of the quandle space $BS_4$ 
of $S_4$. 
We note that $\pi_2(BS_4)$ has been completely computed 
in his subsequent paper \cite{Nos15}. 
\end{remark}

\begin{theorem}
The map $\sigma_* \colon H^Q_3(S_4) \to H^Q_2(S_4)$ 
is a surjective homomorphism $\Z_2 \oplus \Z_4 \to \Z_2$ 
such that $(1,0) \mapsto 1$ and $(0,1) \mapsto 1$.  
\end{theorem}

\begin{proof}
By a direct computation, we have $\sigma(w_1) = z_1$ at the chain level. 
It follows from $\phi (\sigma (w_2)) = 2 \in \Z_4$ that 
$[\sigma (w_2)]$ is non-zero in $H^Q_2(S_4) \cong \Z_2$. 
\end{proof}

\begin{remark}
Although $\sigma(w_2) \neq z_2$ at the chain level, we have $\sigma_*([w_2]) = [z_2]$. 
\end{remark}

\begin{theorem}\label{thm:S_4}
The map $\sigma^* \colon H_Q^2(S_4;\Z_4) \to H_Q^3(S_4;\Z_4)$ is 
an injective homomorphism $\Z_2 \to (\Z_2 \oplus \Z_4) \oplus \Z_2$ 
such that $1 \mapsto (1,2,0)$.  
\end{theorem}

\begin{proof}
We have 
$\sigma^\sharp {\phi} (w_1) = {\phi}(\sigma(w_1)) = {\phi}(z_1) = 2$
and 
$\sigma^\sharp {\phi} (w_2) = {\phi}(\sigma(w_2)) = 2$ in $\Z_4$. 
We also have 
${\eta}_1(w_1) = 0$, ${\eta}_1(w_2) = 2$, 
$\eta_2(w_1) = 1$, $\eta_2(w_2) = 0$ and 
$\eta_{11}(w_1) = \eta_{11}(w_2) = 0$ in $\Z_4$. 
It turns out that $\sigma^* ([{\phi}]) = 
[{\eta}_1] + 2 [\eta_2] + n[\eta_{11}]$ for some $n \in \Z_2$. 
Moreover, for a $2$-cochain 
$f = 2 (
\chi_{(0,3)} + \chi_{(1,0)} + \chi_{(2,3)}
) \in C^2_Q(S_4;\Z_4)$, 
we can check that 
$\sigma^\sharp \phi - \eta_1 - 2 \eta_2 = \delta^2 (f) \in C^3_Q(S_4;\Z_4)$ 
by direct (but somewhat tedious) hand calculation. 
Hence it follows that 
$\sigma^* ([\phi]) = [\eta_1] + 2 [\eta_2] \in H^3_Q(S_4;\Z_4)$. 
\end{proof}


\appendix
\section{Fundamental quandle and fundamental class}\label{sec:fund}

The fundamental classes were explicitly written down 
for classical links \cite{Eis} and surface links \cite{Tan}. 
However 
original quandle homology theory rather than generalizd quandle homology theory 
was used in \cite{Eis}, 
and the notational convention used in \cite{Tan} was based on 
the original one \cite{CEGS} rather than ours \cite{Kam-book}. 
Therefore we review the fundamental classes and also fundamental quandles 
for completeness. 
%
Throughout this appendix, we rely on readers' familiarity with 
basic terminology of diagrams for classical links and surface links.

Let $D$ be a diagram of an oriented classical link (or surface link), and 
$A_D = \{ a_1, \ldots , a_n \}$ the set of all arcs (or sheets) of $D$. 
We assume that each arc (or sheet) of $D$ is given a normal vector to 
indicate the orientation of the link represented by $D$. 
The {\it fundamental quandle} $Q_D$ of $D$ is a quandle generated by 
$A_D$
with the following defining relations. 
At each crossing (or double point curve),  
let $a_j$ be the over arc (or sheet),  
and let $a_i$ and $a_k$ be the under arcs (or sheets) 
such that the normal vector of $a_j$ points from $a_i$ to $a_k$. 
Then the defining relation is given by $a_i * a_j=a_k$ at the crossing 
(or double point curve). 
Hereafter, we basically regard an element of $A_D$ as the element 
of the fundamental quandle $Q_D$.

Let $R_D$ be the set of all (complementary) regions of $D$ in $\R^2$ (or $\R^3$),  
and $r_\infty \in R_D$ the unbounded region of $D$. 
Then it is easy to see that 
there exists a unique map, temporarily denoted by $c^R$ in this paragraph, 
from $R_D$ to $\mathrm{As}(Q_D)$ such that 
\[
c^R(r_\infty) = e 
\quad \text{and} \quad  
c^R(r_{1}) \cdot a = c^R(r_{2})
\]
for each arc (or sheet) $a \in A_D$, where
$r_{1}$ and $r_{2}$ are the regions such that the normal vector of $a$ 
points from $r_{1}$ to $r_{2}$. 
Hereafter, we basically regard an element of $R_D$ as the element 
of the associated group $\mathrm{As}(Q_D)$ through the map $c^R$. 

Let $\tau$ be a crossing (or triple point) of $D$, 
The sign $\varepsilon(\tau) \in \{ +1, -1\}$ is determined by 
whether $\tau$ is positive ($+$) or negative ($-$). 
Among four (or eight) regions around $\tau$,  
there is a unique region such that normal vectors of all of the arcs (or sheets) 
facing the regions point from the region. 
We call the region the \textit{specified region} of $\tau$.

\subsection{Fundamental class of a classical link}
Let $D$ be a diagram of an oriented classical link $L$. 
For a crossing $\tau$ of $D$,  
the \textit{weight} $B(\tau)$ is defined by 
\[
\varepsilon(\tau) (r; x, y) \in C^Q_2(Q_D)_{\mathrm{As}(Q_D)} , 
\]
where $r$ is the specified region of $\tau$, 
$x$ is the under arc facing $r$, and 
$y$ is the over arc at $\tau$. 
Let $|D| \in  C^Q_2(Q_D)_{\mathrm{As}(Q_D)}$ be the sum of the elements $B(\tau)$ 
of all crossings of $D$. 
We can check that $|D| \in Z^Q_2(Q_D)_{\mathrm{As}(Q_D)}$
and define the \textit{fundamental class} 
$[D] \in H^Q_2(Q_D)_{\mathrm{As}(Q_D)}$ 
by its homology class. 
Then this class is independent of the choice of the diagram of $L$ 
in the following sense.

\begin{theorem}\label{thm:1-inv}
For any other diagram $D'$ of the oriented classical link $L$, 
there exists a quandle isomorphism $\alpha \colon Q_D \to Q_{D'}$ 
such that $\alpha_* ([D]) = [D']$, 
where $\alpha_* \colon H^Q_2(Q_D)_{\mathrm{As}(Q_D)} \to 
H^Q_2(Q_{D'})_{\mathrm{As}(Q_{D'})}$ is the induced isomorphism. 
\end{theorem}

\subsection{Fundamental class of a surface link}
Let $D$ be a diagram of an oriented surface link $\mathcal{L}$. 
For a triple point $\tau$ of $D$,  
the \textit{weight} $B(\tau)$ is defined by 
\[
\varepsilon(\tau) (r; x, y,z) \in C^Q_3(Q_D)_{\mathrm{As}(Q_D)} , 
\]
where $r$ is the specified region of $\tau$, 
$x$ is the bottom sheet facing $r$, 
$y$ is the middle sheet facing $r$ and 
$z$ is the top sheet at $\tau$. 
Let $|D| \in  C^Q_3(Q_D)_{\mathrm{As}(Q_D)}$ be the sum of the elements $B(\tau)$ 
of all triple points of $D$. 
We can check that $|D| \in Z^Q_3(Q_D)_{\mathrm{As}(Q_D)}$
and define the \textit{fundamental class} 
$[D] \in H^Q_3(Q_D)_{\mathrm{As}(Q_D)}$ 
by its homology class. 
Then this class is independent of the choice of the diagram of $\mathcal{L}$ 
in the following sense.

\begin{theorem}\label{thm:2-inv}
For any other diagram $D'$ of the oriented surface link $\mathcal{L}$, 
there exists a quandle isomorphism $\alpha \colon Q_D \to Q_{D'}$ 
such that $\alpha_* ([D]) = [D']$, 
where $\alpha_* \colon H^Q_3(Q_D)_{\mathrm{As}(Q_D)} \to 
H^Q_3(Q_{D'})_{\mathrm{As}(Q_{D'})}$ is the induced isomorphism. 
\end{theorem}


\section*{Acknowledgments}
The authors would like to thank Katsumi Ishikawa for 
sharing his diagrammatic idea for the proof of the main theorems.  
The statements and proofs of Proposition~\ref{prop:shadow-1} and~\ref{prop:2-knot}, 
which are keys for the proofs of Theorem~\ref{thm:shadow-1} and~\ref{thm:2-knot},  
could be formulated by interpreting his idea into algebraic language. 
It would be greatful that he kindly found 
the $2$-cochain appeared in the proof of Theorem~\ref{thm:S_4} 
by using his computer. 
They also would like to thank Takefumi Nosaka for telling them 
known results on (co)homology groups of Alexander quandles, 
and Sukuse Abe for useful comments on explicit forms of 
$4$-cocycles for Alexander quandles. 
The second-named author has been supported in part by 
the Grant-in-Aid for Scientific Research (C), (No.~JP17K05242), 
Japan Society for the Promotion of Science.


\bibliographystyle{amsplain}

\providecommand{\bysame}{\leavevmode\hbox to3em{\hrulefill}\thinspace}
\providecommand{\MR}{\relax\ifhmode\unskip\space\fi MR }
\providecommand{\MRhref}[2]{%
  \href{http://www.ams.org/mathscinet-getitem?mr=#1}{#2}
}
\providecommand{\href}[2]{#2}


\end{document}